\documentclass[leqno,11pt]{amsart}
\usepackage{amsmath, amsfonts}
\usepackage{amssymb}
\usepackage{graphicx}
\usepackage[usenames]{color}
\usepackage{amsxtra}

\newtheorem{theo}{Theorem}[section]
\newtheorem{defi}[theo]{Definition}
\newtheorem{lemm}[theo]{Lemma}
\newtheorem{prop}[theo]{Proposition}
\newtheorem{rem}[theo]{Remark}

\usepackage{amsmath}

\usepackage{amssymb}
\usepackage{graphicx}

\begin{document}

\title[]{On the convergence of renormalizations of piecewise smooth homeomorphisms on the circle}




\author{Abdumajid Begmatov and Kleyber Cunha}

\bigskip

\address{Abdumajid Begmatov, \newline  Institute of Mathematics, Academy of Science of the Republic of Uzbekistan, Do'rmon yo'li street 29, Akademgorodok, 100125 Tashkent, Uzbekistan. E-mail: \textcolor{blue}{abdumajidb@gmail.com}}

\bigskip

\address{Kleyber Cunha, \newline Departamento de Matem\'{a}tica, Universidade Federal da Bahia, Av. Ademar de Barros s/n, CEP 40170-110, Salvador, BA, Brazil.  E-mail: \textcolor{blue}{kleyber@ufba.br}}

\begin{abstract}
We study renormalizations of piecewise smooth homeomorphisms on the circle, by considering such maps as generalized interval exchange maps of genus one. Suppose that $Df$ is absolutely continuous on each interval of continuity and $D\ln{Df}\in \mathbb{L}_{p}$ for some $p>1$. We prove, that under certain combinatorial assumptions on $f_{1}$ and $f_{2}$, corresponding renormalizations approach to each other in $C^{1+L_{1}}$-norm.
\end{abstract}

\maketitle

\sloppy

\section{Introduction}

One of the most important problems of theory of dynamical systems is to study regularity of conjugating map between two dynamical systems, which a priory are only topologically equivalent. In the context of circle dynamics topologically equivalence means that for two maps with the same irrational rotation number and the same local structure of their singular points, there exists a circle homeomorphism conjugating these maps. Notice that statements on the regularity of conjugation can be obtained by using the convergence of renormalizations of circle homeomorphisms. In the case of sufficiently smooth circle diffeomorphisms, renormalizations converge exponentially fast to a family of linear maps with slope 1. Such a convergence together with the condition on the rotation number (of Diophantine type) imply the regularity of conjugation (see \cite{KS1989}, \cite{KO1989.1}, \cite{KO1989.2}). Renormalizations of circle homeomorphisms with a break point converge exponentially fast to a family of fractional linear maps \cite{KV1991}. Investigation of the fractional linear maps in \cite{KhKhm2003},
\cite{KT2013} and \cite{KhYam} showed, that the renormalization operator in that space possesses hyperbolic properties. Using convergence of renormalizations of two topologically equivalent circle homeomorphisms with a break point proved, that the conjugation between these maps is $C^{1}$- smooth (see  \cite{KhK2013}, \cite{KhK2014}, \cite{KT2013}).

Circle homeomorphisms with several break points can be considered as generalized interval exchange transformations of genus one. Marmi, Moussa and Yoccoz introduced in \cite{MMY2012} generalized interval exchange transformations, obtained by replacing the affine restrictions of generalized interval exchange transformations in each subinterval with smooth diffeomorphisms. They showed that sufficiently smooth generalized interval exchange transformations of a certain combinatorial type, which are deformations of standard interval exchange transformations and tangent to them at the points of discontinuities, are smoothly linearizable. Cunha and Smania studied in \cite{CS2013} Rauzy-Veech renormalizations of piecewise $C^{2+\nu}$- smooth circle homeomorphisms by considering such maps as generalized interval exchange transformations of genus one. They proved that Rauzy-Veech renormalizations of $C^{2+\nu}$- smooth generalized interval exchange maps satisfying a certain combinatorial condition are approximated by piecewise M\"{o}bius transformations in $C^2$- norm. Using convergence of renormalizations of two generalized interval exchange maps with the same bounded-type combinatorics and zero mean nonlinearities they proved in \cite{CS2014} that these maps $C^{1}$-smoothly conjugate to each other.

We considered in \cite{BDM2014}, \cite{CDB2017} the bottom of the scale of smoothness for a homeomorphism $f$, that is, $Df$ is absolutely continuous on each interval of continuity and $D\log Df\in \mathbb{L}_{p}$ for some $p>1$. The latter conditions on smoothness of $f$, are called the Katznelson and Ornstein's (KO, for short) smoothness condition \cite{KO1989.2}. For this low smoothness case in \cite{CDB2017} it was shown, that the Rauzy-Veech renormalizations of piecewise KO-smooth maps that satisfy certain combinatorial assumptions approach piecewise M\"{o}bius maps in the $C^{1+L_{1}}$ norm. Note that KO smoothness condition is smaller than $C^{2+\nu}$ smooth case considered in \cite{CS2013}, but the obtained convergence rate is slower than exponential.

The purpose of the present work is to study the convergence of renormalizations of two topological equivalent generalized interval exchange maps. We prove that Rauzy-Veech renormalizations of two topological equivalent generalized interval exchange maps of genus one and of KO smoothness approach to each other in $C^{1+L_{1}}$-norm.

The structure of the paper is as follows. In Section 2 we define Rauzy-Veech renormalization of interval exchange maps and define a sequence of dynamical partition associated renormalization map. In Section 3 we formulate our main results. In Section 4 we define Rauzy-Veech cocycle and formulate theorems on the hyperbolicity of the cocycle. Finally in Section 5 we prove our main theorems.

\section{Preliminaries and notations}

\textbf{\large{Rauzy-Veech renormalization.}} This work concerns renormalization of generalizad interval exchange maps. Let $I$ be an open bounded interval and $\mathcal{A}$ be an alphabet with $d\geq 2$ symbols. Consider the partition of $I$ into $d$ subintervals indexed by $\mathcal{A}$, that is, $\mathcal{P}=\{I_{\alpha},\,\, \alpha\in \mathcal{A}\}$. Let $f: I\rightarrow I$ be a bijection. We say that the triple $(f, \mathcal{A}, \mathcal{P})$ is a \textbf{generalized interval exchange map} with $d$ intervals (for short g.i.e.m.), if $f|_{I_\alpha}$ is an orientation-preserving homeomorphism for all $\alpha\in\mathcal{A}.$ Here and later, all intervals will be bounded, closed on the left and open on the right.

If $f|_{I_\alpha}$ is a translation and $|f(I_{\alpha})|=|I_{\alpha}|$, then $f$ is called a \textbf{standard interval exchange map} (for short s.i.e.m.). When $d=2$, identifying the endpoints of $I$, s.i.e.m. correspond to linear rotations of the circle and g.i.e.m. correspond to circle homeomorphisms.

Now we formulate some conditions on the combinatorics for g.i.e.m and define the renormalization scheme. Note that the combinatorial conditions and the renormalization scheme are the same for generalized and standard i.e.m. cases.

The order of the subintervals $I_{\alpha}$ before and after the map, constitutes the combinatorial data for $f$, which will be explicitly defined as follows.

Given two intervals $J$ and $U$, we will write $J <U$, if their interiors are disjoint and $x<y$, for every $x\in J$ and $y\in U$. This defines a partial order in the set of all intervals.

Let $f: I\rightarrow I$ be a g.i.e.m. with alphabet $\mathcal{A}$ and $\pi_0,\, \pi_1:\mathcal{A}\rightarrow \{1, . . . , d\}$, be bijections such that
$$
\pi_0(\alpha)<\pi_0(\beta),\,\,\,\,\,\, \mbox{\rm iff} \,\,\,\,\,\,\,\,I_{\alpha}<I_{\beta},
$$
and
$$
\pi_1(\alpha)<\pi_1(\beta),\,\,\,\,\,\, \mbox{\rm iff} \,\,\,\,\,\,\,\,f(I_{\alpha})<f(I_{\beta}).
$$

We call pair $\pi=(\pi_0, \pi_1)$ the \textbf{combinatorial data} associated to the g.i.e.m. $f$. We call $p=\pi_1^{-1}\circ\pi_0: \{1,...,d\}\rightarrow \{1,...,d\}$ the \textbf{monodromy invariant} of the pair $\pi=(\pi_0, \pi_1)$. When appropriate we will also use the notation $\pi = (\pi(1),\, \pi(2),\,...,\pi(d))$ for the combinatorial data of $f$. We always assume that the pair $\pi=(\pi_0, \pi_1)$ is \textbf{irreducible}, that is, for all $j\in \{1,...,d-1\}$ we have: $\pi_0^{-1}({1, . . . , j})\neq \pi_1^{-1}({1, . . . , j})$.

Let $\pi=(\pi_0, \pi_1)$ be the combinatorial data associated to the g.i.e.m $f$. For each $\varepsilon\in \{0, 1\}$, denote by $\alpha(\varepsilon)$ the last symbol in the expression of $\pi_{\varepsilon}$, that is \, $\alpha(\varepsilon)=\pi^{-1}_{\varepsilon}(d)$.

Let us assume that the intervals $I_{\alpha(0)}$ and $f(I_{\alpha(1)})$ have different lengths. Then the g.i.e.m. $f$ is called \textbf{Rauzy-Veech renormalizable}(renormalizable, for short). If $|I_{\alpha(0)}|>|f(I_{\alpha(1)})|$ we say that $f$ is renormalizable of \textbf{type} $\textbf{0}$. When $|I_{\alpha(0)}|<|f(I_{\alpha(1)})|$ we say that $f$ is renormalizable of \textbf{type} $\textbf{1}$.  In either case, the letter corresponding to the largest of these intervals is called \textbf{winner} and the one corresponding to the shortest is called the \textbf{loser} of $\pi$. Let $I^{(1)}$ be the subinterval of $I$ obtained by removing the loser, that is, the shortest of these two intervals:
$$
I^{(1)}=\left\{\begin{array}{ll}
I\setminus f(I_{\alpha(1)}),\,\,\, \mbox{\rm if} \,\,\,\,\, \mbox{\rm type 0}, \\
I\setminus I_{\alpha(0)},\,\,\, \mbox{\rm if} \,\,\,\,\, \mbox{\rm
type 1}.
\end{array}\right.
$$

Since the loser is the last subinterval on the right of $I$, the intervals $I$ and $I^{(1)}$ have the same left endpoint.

The \textbf{Rauzy-Veech induction} of $f$ is the first return map $R(f)$ to the subinterval $I^{(1)}$. We want to see $R(f)$ is again g.i.e.m. with the same alphabet $\mathcal{A}$. For this we need to associate to this map an $\mathcal{A}$ - indexed partition of its domain. Denote by $I^{(1)}_{\alpha}$ the subintervals of $I^{(1)}$. Let $f$ be renormalizable of type $0$. Then the domain of $R(f)$ is the interval $I^{(1)}=I\setminus f(I_{\alpha(1)})$ and
we have
\begin{equation}\label{Itype0}
I^{(1)}_{\alpha}=\left\{\begin{array}{ll}
I_{\alpha},\,\,\,\,\,\,\, \mbox{\rm for} \,\,\,\,\, \alpha\neq\alpha(0), \\
I_{\alpha(0)}\setminus f(I_{\alpha(1)}),\,\,\, \mbox{\rm for}
\,\,\,\,\, \alpha=\alpha(0).
\end{array}\right.
\end{equation}
These intervals form a partition of the interval $I^{(1)}$ and denoted by $\mathcal{P}^{(1)}=
\{I^{(1)}_{\alpha},\, \alpha\in \mathcal{A}\}$. Since $f(I_{\alpha(1)})$ is the last interval on the right of
$f(\mathcal{P})$, we have $f(I^{(1)}_{\alpha})\subset I^{(1)}$ for every $\alpha\neq \alpha(1)$. This means that, $R(f):=f$ restricted to these $I^{(1)}_{\alpha}$. On the other hand, due to $I^{(1)}_{\alpha(1)}=I_{\alpha(1)}$, we have
$$
f\left(I^{(1)}_{\alpha(1)}\right)=f\left(I_{\alpha(1)}\right)\subset I_{\alpha(0)},\,\,\,\, \mbox{\rm and so} \,\,\,\,
f^2\left(I^{(1)}_{\alpha(1)}\right)\subset f\left(I_{\alpha(0)}\right)\subset I^{(1)}.
$$
Then $R(f):=f^2$ restricted to $I^{(1)}_{\alpha(1)}$. Thus,
\begin{equation}\label{Rtype0}
R(f)(x)=\left\{\begin{array}{ll}
f(x),\,\,\,\,\, \mbox{\rm if} \,\,\,\,\, x\in I^{(1)}_{\alpha}\,\,\,\text{ and } \alpha\neq\alpha(1), \\
f^2(x),\,\,\,\, \mbox{\rm if} \,\,\,\,\, x\in I^{(1)}_{\alpha(1)}.
\end{array}\right.
\end{equation}

If $f$ is renormalizable of type $1$, the domain of $R(f)$ is the interval \ $I^{(1)}=I\setminus I_{\alpha(0)}$ and we have
\begin{equation}\label{Itype1}
I^{(1)}_{\alpha}=\left\{\begin{array}{lll}
I_{\alpha},\,\,\,\,\,\,\, \mbox{\rm for} \,\,\,\,\, \alpha\neq\alpha(0), \alpha(1), \\
f^{-1}(I_{\alpha(0)}),\,\,\, \mbox{\rm for} \,\,\,\,\,
\alpha=\alpha(0), \\
I_{\alpha(1)}\setminus f^{-1}(I_{\alpha(0)}),\,\,\, \mbox{\rm for}
\,\,\,\,\, \alpha=\alpha(1).
\end{array}\right.
\end{equation}
Then $f\left(I^{(1)}_{\alpha}\right)\subset I^{(1)}$ for every $\alpha\neq \alpha(0)$, and so $R(f)=f$ restricted to these $I^{(1)}_{\alpha}$. On the other hand,
$$
f^2\left(I^{(1)}_{\alpha(0)}\right)= f(I_{\alpha(0)})\subset I^{(1)},
$$
and, so $R(f)=f^2$ restricted to $I^{(1)}_{\alpha(0)}$. Thus,
\begin{equation}\label{Rtype1}
R(f)(x)=\left\{\begin{array}{ll}
f(x),\,\,\,\,\, \mbox{\rm if} \,\,\,\,\, x\in I^{(1)}_{\alpha}\,\,\,\text{ and } \alpha\neq\alpha(0), \\
f^2(x),\,\,\,\, \mbox{\rm if} \,\,\,\,\, x\in I^{(1)}_{\alpha(0)}.
\end{array}\right.
\end{equation}
It is easy to see, that $R(f)$ is a bijection on $I^{(1)}$ and an orientation-preserving homeomorphisms on each $I^{(1)}_{\alpha}$. Moreover, the alphabet $\mathcal{A}$  for $f$ and $R(f)$ remains the same.

The triple $(R(f), \mathcal{A}, \mathcal{P}^1)$ is called the \textbf{Rauzy-Veech renormalization} of $f$. If $f$ is
renormalizable of type $\varepsilon\in \{0, 1\}$, then the combinatorial data $\pi^1=(\pi_0^1, \pi_1^1)$ of $R(f)$ are given by
$$
\pi_{\varepsilon}^1:=\pi_{\varepsilon},\,\,\,\,\,\, \mbox{\rm
and}\,\,\,\,\,\,\,
\pi_{1-\varepsilon}^1(\alpha)=\left\{\begin{array}{lll}
\pi_{1-\varepsilon}(\alpha),\,\,\,\,\, \mbox{\rm if} \,\,\,\,\, \pi_{1-\varepsilon}(\alpha)\leq \pi_{1-\varepsilon}(\alpha(\varepsilon)), \\
\pi_{1-\varepsilon}(\alpha)+1,\,\,\,\,\, \mbox{\rm if} \,\,\,\,\, \pi_{1-\varepsilon}(\alpha(\varepsilon))< \pi_{1-\varepsilon}(\alpha)<d, \\
\pi_{1-\varepsilon}(\alpha(\varepsilon))+1,\,\,\,\,\, \mbox{\rm if}
\,\,\,\,\, \pi_{1-\varepsilon}(\alpha)=d.
\end{array}\right.
$$

We say that a g.i.e.m. $f$ is \textbf{infinitely renormalizable}, if $R^{n}(f)$ is well defined for every $n\in \mathbb{N}$. Let $I^{(n)}$ be the domain of $R^{n}(f)$. It is clear that, $R^{n}(f)$ is the first return map for $f$ to the interval $I^{(n)}$. Similarly, $R^{n}(f)^{-1}=R^{n}(f^{-1})$ is the first return map for $f$ to the
interval $I^{(n)}$.

For every interval of the form $J=[a, b)$ we put $\partial J:=\{a\}$.
\begin{defi}
We say that g.i.e.m. $f$  \textbf{has no connection}, if
\begin{equation}\label{Keane}
f^m(\partial I_{\alpha})\neq \partial I_{\beta},\,\,\,\,\,\,\,
\mbox{\it for all} \,\,\,\,\,\, m\geq 1\,\,\,\, \mbox{\it
and}\,\,\,\,\, \alpha,\, \beta\in\mathcal{A} \,\,\,\,\,\,\,\,\,
\mbox{\it with} \,\,\,\,\,\, \pi_{0}(\beta)\neq 1.
\end{equation}
\end{defi}

It is clear that in case $\pi_0(\beta)=1$ then $f(\partial I_{\alpha})=\partial I_{\beta}$ for $\alpha=\pi_1^{-1}(1)$. Condition (\ref{Keane}) is called  the \textbf{Keane condition}. Keane \cite{Ke1975} showed that the no connection condition is a necessary and sufficient condition for $f$ to be infinitely renormalizable. Condition (\ref{Keane}) means that the orbits of the left end point of the subintervals $I_{\alpha},\,\alpha\in\mathcal{A}$ are disjoint when ever they can be.

Let $\varepsilon_{n}$ be the type of the $n$-th renormalization and let $\alpha_n(\varepsilon_{n})$ the winner and
$\alpha_n(1-\varepsilon_{n})$ be the loser of the $n$-th renormalization.

\begin{defi}
We say that g.i.e.m. $f$ has $k$- \textbf{bounded combinatorics},
if for each $n\in \mathbb{N}$ and $\beta,\, \gamma\in\mathcal{A}$
there exist $n_1, p\geq0$ with $|n-n_1|<k$ and $|n-n_1-p|<k$ such
that
$$\alpha_{n_1}(\varepsilon_{n_1})=\beta,\,
\alpha_{n_1+p}(1-\varepsilon_{n_1+p})=\gamma,\,\,\, \mbox{\it and}
$$
$$
\alpha_{n_1+i}(1-\varepsilon_{n_1+p})=\alpha_{n_1+i+1}(\varepsilon_{n_1+i}),\,\,\,\,\,\,\,\,
\mbox{\it for every} \,\,\,\,\,\, 0\leq i<p.
$$
\end{defi}

We say that g.i.e.m. $f: I\rightarrow I$ has \textbf{genus one}(or belongs to the \textbf{rotation class}), if $f$ has at most two discontinuities. Note that every g.i.e.m. with either two or three intervals has genus one. The genus of g.i.e.m. is invariant under renormalization.

\textbf{\large{Dynamical partition.}} Using Rauzy-Veech induction we define dynamical partition of the interval $I$. Let $(f, \mathcal{A}, \mathcal{P})$ be a g.i.e.m. with $d$ intervals and $\mathcal{P}=\{I_{\alpha}:
\alpha\in \mathcal{A}\}$ be the initial $\mathcal{A}$- indexed partition of $I$. For specificity we take $I=[0, 1)$. Suppose that $f$ is infinitely renormalizable. Let $I^{(n)}$ be the domain of $R^n(f)$. Note that $I^{(n)}$ is the nested sequence of subintervals, with the same left endpoint of $I$. We want to construct the dynamical partition of $I$ associated to the domain of $R^n(f)$.

As mentioned above, $R(f)$ is g.i.e.m. with $d$ intervals and the intervals $I^{(1)}_{\alpha}$ generate an $\mathcal{A}$- indexed partition of $I^{(1)}$, denoted by $\mathcal{P}^1$. By induction one can check, that $R^n(f)$ is g.i.e.m. with $d$ intervals. Let $\mathcal{P}^{n}=\{I^{(n)}_{\alpha}: \alpha\in \mathcal{A}\}$ be the $\mathcal{A}$- indexed partition of $I^{(n)}$, generated by $R^{n}(f)$. We call $\mathcal{P}^{n}$ the {\bf fundamental partition} and $I^{(n)}_{\alpha}$ the \textbf{fundamental segments} of rank $n$.

Since $R^n(f)$ is the first return map for $f$ to the interval $I^{(n)}$, each fundamental segment $I^{(n)}_{\alpha}\in \mathcal{P}^{n}$ returns to $I^{(n)}$ under certain iterates of the map $f$. Until returning, these intervals will be in the interval $I\setminus I^{(n)}$ for some time. Consequently the system of
intervals (their interiors are mutually disjoint)
$$
\xi_n=\left\{ f^{i}(I^{(n)}_{\alpha}),\,\, 0\leq i\leq q_{\alpha}^{n}-1,\,\, \alpha\in \mathcal{A}\right\}
$$
cover the whole interval and form a partition of $I$.

The system of intervals $\xi_n$ is called the \textbf{n-th dynamical partition} of $I$. The dynamical partitions $\xi_n$ are refined with increasing $n$, where $\xi_{n+1}\supset\xi_{n}$ means that
any element of the preceding partition is a union of a number of elements of the next partition, or belongs to the next partition. Denote by $\xi^{pr}_{n+1}$ the system of preserved intervals of $\xi_{n}$. More precisely, if $R^nf$ has type 0
$$\xi^{pr}_{n+1}=\left\{ f^{i}(I^{(n)}_{\alpha}),\,\, 0\leq i\leq q^{n}_{\alpha}-1,\,\, \mbox{~for~} \alpha\not=\alpha(0)\right\},$$
and if $R^nf$ has type 1

$$\xi^{pr}_{n+1}=\left\{ f^{i}(I^{(n)}_{\alpha}),\,\, 0\leq i\leq q^{n}_{\alpha}-1,\,\, \mbox{~for~} \alpha\not=\alpha(1)\right\}.$$

Let $\xi^{tn}_{n+1}:=\xi_{n+1}\setminus \xi^{pr}_{n+1}$ be the set of elements of $\xi_{n+1}$ which are properly contained in some element of $\xi_n.$ Therefore if $R^nf$ has type 0

\begin{eqnarray*}
\xi^{tn}_{n+1}&=&\left\{ f^i(I^{(n+1)}_{\alpha(0)}),\,\, 0\leq i< q^{n}_{\alpha(0)}\right\}\bigcup \left\{ f^i(I^{(n+1)}_{\alpha(1)}),\,\, 0\leq i< q^{n}_{\alpha(0)}\right\}\\
&=&\bigcup_{i=0}^{q^{n}_{\alpha(0)}-1}\left\{f^i\left(I^{(n)}_{\alpha(0)}\setminus f^{q^n_{\alpha(1)}}I^{(n)}_{\alpha(1)}\right)\right\}\bigcup \bigcup_{i=q^n_{\alpha(1)}}^{q^n_{\alpha(1)}+q^n_{\alpha(0)}-1}\left\{ f^i(I^{(n)}_{\alpha(1)})\right\},
\end{eqnarray*}

and if $R^nf$ has type 1

\begin{eqnarray*}
	\xi^{tn}_{n+1}&=&\left\{f^i(I^{(n+1)}_{\alpha(0)}),\,\, 0\leq i< q^{n}_{\alpha(1)}\right\}\bigcup\left\{f^i(I^{(n+1)}_{\alpha(1)}),\,\, 0\leq i< q^{n}_{\alpha(1)}\right\}\\
	& = & \bigcup_{i=0}^{q^n_{\alpha(1)}-1}\left\{f^i\left(f^{-q^n_{\alpha(1)}}(I^{(n)}_{\alpha(0)})\right)\right\}\bigcup \bigcup_{i=0}^{q^n_{\alpha(1)}-1}\left\{f^i\left(I^{(n)}_{\alpha(1)}\setminus f^{-q^n_{\alpha(1)}}(I^{(n)}_{\alpha(0)})\right)\right\}.
\end{eqnarray*}

So, the partition $\xi_{n+1}$ consists of the \emph{preserving} elements of $\xi_{n}$ and the \emph{images of two (new) intervals} for defining $R^{n+1}(f)$, that is, $\xi_{n+1}=\xi^{pr}_{n+1}\cup \xi^{tn}_{n+1}$. Note also that for the first return time $q^{n}_{\alpha}$, we have:
\begin{itemize}
\item[(1)] if \ $\alpha=\alpha^{n}(\varepsilon)$, then \ $q^{n+1}_{\alpha^{n}(\varepsilon)}=q^{n}_{\alpha^{n}(\varepsilon)}$;
\item[(2)] if \ $\alpha=\alpha^{n}(1-\varepsilon)$, then \ $q^{n+1}_{\alpha^{n}(1-\varepsilon)}=q^{n}_{\alpha^{n}(1-\varepsilon)}+q^{n}_{\alpha^{n}(\varepsilon)}$.
\end{itemize}

\bigskip

The following lemma will be used in the proof of our main results in Section 5.

\begin{lemm}\label{extra} Let $\{r_{n}\}\in l_2$ be a  sequence of positive numbers and $\lambda\in (0, 1)$. Set
$$
x_n=\sum\limits_{j=n}^{\infty}\lambda^{j-n}r_j,\,\,\,\, y_{n}=\sum\limits_{j=0}^{n-1}\lambda^{j}r_{n+j},\,\,\,\,\,\,\, z_{n}=\sum\limits_{j=0}^{n-1}\lambda^{j}r_{n-j}.
$$
Then the sequences $\{x_{n}\}, \{y_{n}\}$ and $\{z_{n}\}$ belong to $l_2$.
\end{lemm}
\begin{proof} Using H\"{o}lder's inequality for sum we get
\begin{equation}\label{xn2}
x^{2}_n=\left(\sum\limits_{j=n}^{\infty}\lambda^{\frac{j-n}{2}}\cdot\lambda^{\frac{j-n}{2}}r_j\right)^{2}\leq \left(\sum\limits_{j=n}^{\infty}\lambda^{j-n}\right)\cdot\left(\sum\limits_{j=n}^{\infty}\lambda^{j-n}r^{2}_j\right)=
\frac{1}{1-\lambda}\sum\limits_{j=n}^{\infty}\lambda^{j-n}r^{2}_j.
\end{equation}
An elementary calculations shows that
$$
\sum\limits_{n=1}^{\infty}x^{2}_n=\frac{1}{1-\lambda}\sum\limits_{n=1}^{\infty}\sum\limits_{j=n}^{\infty}\lambda^{j-n}r^{2}_{j}=
\frac{1}{1-\lambda}\left(\sum\limits_{n=1}^{\infty}\lambda^{n-1}\right)\left(\sum\limits_{n=1}^{\infty}r^{2}_{n}\right)\leq \frac{1}{(1-\lambda)^{2}}\sum\limits_{n=1}^{\infty}r^{2}_{n}.
$$
The last inequality and condition $\{r_{n}\}\in l_2$ imply that $\{x_{n}\}\in l_{2}$. It is easy to check that $y_{n}\leq x_{n}$. Consequently, we have $\{y_{n}\}\in l_{2}$. Using analogously simplifications as in the inequality (\ref{xn2}) we get: $z^{2}_{n}\leq (1-\lambda)^{-1}u_{n}$, where $u_{n}=\sum\limits_{j=1}^{n}\lambda^{n-j}r^{2}_{j}$. \, In order to estimate the series $\sum\limits_{n=1}^{\infty}u_{n}$ we pay attention to the sequence $u_{n}$:

$\,\,\,\,\,\,\,\,\, u_{1}=r^{2}_{1}$

$\,\,\,\,\,\,\,\,\, u_{2}=\lambda r^{2}_{1}+r^{2}_{2}$

$\,\,\,\,\,\,\,\,\, u_{3}=\lambda^{2} r^{2}_{1}+\lambda r^{2}_{2}+r^{2}_{3}$

$\,\,\,\,\,\,\,\,\,\,\,\,\,\,\,\,\,\,  \cdot \cdot \cdot \cdot \cdot \cdot \cdot \cdot \cdot \cdot \cdot \cdot \cdot$

$\,\,\,\,\,\,\,\,\, u_{n}=\lambda^{n-1} r^{2}_{1}+\lambda r^{n-2}_{2}+\cdot\cdot\cdot+\lambda r^{2}_{n-1}+r^{2}_{n}$

$\,\,\,\,\,\,\,\,\,\,\,\,\,\,\,\,\,\,  \cdot \cdot \cdot \cdot \cdot \cdot \cdot \cdot \cdot \cdot \cdot \cdot \cdot \cdot \cdot \cdot$\\
Summarizing the right sides of the last expressions by diagonal we get
$$
\sum\limits_{n=1}^{\infty}u_{n}\leq \frac{1}{1-\lambda}\sum
\limits_{n=1}^{\infty}r^{2}_{n},\,\,\,\,\,\, \mbox{\rm and consequently}\,\,\,\,\,\, \sum\limits_{n=1}^{\infty}z^{2}_{n}\leq\frac{1}{(1-\lambda)^{2}}\sum\limits_{n=1}^{\infty}r^{2}_{n},
$$
The last inequality and condition $\{r_{n}\}\in l_2$ imply that $\{z_{n}\}\in l_2$.
\end{proof}

\section{Previous results and the statement of the main theorems}

Denote by $\mathbb{B}^{KO}$  the set of g.i.e.m. satisfying the following conditions:
\begin{itemize}

\item[(i)] the map $f$ has genus one (cyclic permutation);

\item[(ii)] the map $f$ has no connection and has $k$- bounded combinatorics;

\item[(iii)] for each $\alpha\in \mathcal{A}$ we can extend $f$ to $\overline{I}_{\alpha}$ as an orientation-preserving diffeomorphism satisfying Katznelson and Ornstein's(KO, for short) smoothness condition: $f'$ is absolutely continuous and $f''\in L_{p}$, for some $p>1$.
\end{itemize}

Denote by $\mathbb{B}^{KO}_{\star}$  the subset of functions $f\in \mathbb{B}^{KO}$ satisfying \textbf{zero mean nonlinearity} condition:
$$
\int\limits_{[0, 1]}\frac{f''(t)}{f'(t)}dt=0.
$$

The main idea of the renormalization group method is to study the behaviour of the renormalization map $R^{n}(f)$ as $n\rightarrow\infty$. For this usually rescaling of the coordinates is used.

For any homeomorphism $g: I\rightarrow J$ define the \textbf{Zoom}(or renormalized coordinate) $Z_{H}(g)=\tau^{-1}\circ g \circ \tau$,\,
where $\tau: [0, 1]\rightarrow I$ is orientation-preserving affine map. Define the fractional linear transformations $F_{n}: [0, 1]\rightarrow [0, 1]$ as follows:
\begin{equation}\label{Fn}
F_{n}(x)=\frac{xm_{n}}{1+x(m_{n}-1)},\,\,\,\, \mbox{\rm where} \,\,\,\,\, m_{n}=\exp\{-\sum\limits_{i=0}^{q^{\alpha}_{n}-1}\int_{I^{(n)}_{\alpha}}\frac{f''(t)}{2f'(t)}dt\}.
\end{equation}

Whenever necessary, we will use $D^{m}f$ instead of the $m^{th}$ derivative of $f$. The following result states about the behaviour of renormalization of a g.i.e.m.

\begin{theo}\label{main1}(see \cite{CDB2017}) $(1)$ Let $f\in \mathbb{B}^{KO}$. Then for all $\alpha\in
\mathcal{A}$ the following bounds hold:
$$
\|Z_{I^{(n)}_{\alpha}}(R^{n}(f))-F_{n}\|_{C^1[0,
1])}\leq \delta_{n},\,\,\,\,\,\,\,\,
\|Z_{I^{(n)}_{\alpha}}(D^2R^{n}(f))-D^2F_{n}\|_{L_1([0,
1], d\ell)}\leq \delta_{n}.
$$
$(2)$ Let $f\in \mathbb{B}^{KO}_{\star}$. Then for all $\alpha\in
\mathcal{A}$ the following bounds hold:
$$
\|Z_{I^{(n)}_{\alpha}}(R^{n}(f))-Id\|_{C^1[0, 1])}\leq
\delta_{n},\,\,\,\,\,\,\,\,
\|Z_{I^{(n)}_{\alpha}}(D^2R^{n}(f))\|_{L_1([0, 1], d\ell)}\leq
\delta_{n},
$$
where $\delta_{n}=\mathcal{O}(\lambda^{\sqrt{n}}+\eta_{n}),\,\, \lambda\in (0, 1)$ and $\eta_{n}\in l_{2}$.
\end{theo}

The reasons for the restriction "k-bounded combinatorics" in the condition $(ii)$ are technical. We conjecture that the statement of Theorem \ref{main1} holds for all combinatorics keeping no connection and genus one conditions.

Our main results in the present paper are the following

\begin{theo}\label{main3} Let $f\in \mathbb{B}^{KO}$. Then there exist a sequence of positive numbers $\{\delta_{n}\}\in l_{2}$ and an affine i.e.m. $(f_{A}, \mathcal{A}, \{\widetilde{I}_{\alpha}\}_{\alpha\in \mathcal{A}})$, that is, $f_{A}|_{\widetilde{I}_{\alpha}}$ is affine for each $\alpha\in
\mathcal{A}$ such that
\begin{itemize}
\item[(i)]  $f_{A}$ has the same combinatorics of $f$;
\item[(ii)] $\|R^{n}f-R^{n}f_{A}\|_{C^1([0, 1])}\leq \delta_{n},\,\,\,\, \|D^{2}R^{n}f-D^{2}R^{n}f_{A}\|_{L_1([0, 1], d\ell)}\leq \delta_{n}$.
\end{itemize}
\end{theo}

\begin{theo}\label{main4}(Universality). Suppose that $f$ and $g$ satisfy the assumptions of Theorem \ref{main3} and they have the same
combinatorics, and they are break-equivalents. Then one can choose that $f_{A}=g_{A}$.
\end{theo}

The next theorem is a consequence of Theorem \ref{main3} and \ref{main4}.

\begin{theo}\label{main5} Let $f, g\in \mathbb{B}^{KO}_{\star}$ be such that
\begin{itemize}
\item[(i)]  $f$ and $g$ have the same combinatorics;
\item[(ii)] $f$ and $g$ are break equivalent.
\end{itemize}
Then there exists a sequence of positive numbers $\{\delta_{n}\}\in l_{2}$ such that
$$
\|R^{n}f-R^{n}g\|_{C^1([0, 1])}\leq \delta_{n},\,\,\,\, \|D^{2}R^{n}f-D^{2}R^{n}g\|_{L_1([0, 1], d\ell)}\leq \delta_{n}.
$$
\end{theo}

\smallskip

Note that the KO smoothness condition is smaller than $C^{2+\nu}$ smoothness considered in \cite{CS2013}. But the rate of approximations in Theorems \ref{main5} is not exponential, contrary to the $C^{2+\nu}$ smoothness case. We believe this rate of approximations will be suffice to prove the absolute continuity of conjugating map between $f$ and $g$.

\section{Rauzy--Veech cocycle}

The sequence of renormalizations $R^{n}(f)$ can also be considered as the action of Rauzy--Veech cocycle $\Theta^{n}$ on the space of i.e.m. which will be defined in this section. We follow the presentations of \cite{CS2014} and \cite{Viana}.

Consider an interval exchange map $(f,\, \mathcal{A},\, \{I_{\alpha}\})$. Denote by $\lambda=(\lambda_{\alpha})_{\alpha\in \mathcal{A}}$ the lengths of subintervals $\lambda_{\alpha}:=|I_{\alpha}|$. It is well-known that any standard interval exchange maps $f$ is uniquely defined by combinatorial data $\pi=(\pi_{0},\, \pi_{1})$ and by length vector $\lambda=(\lambda_{\alpha})_{\alpha\in \mathcal{A}}$. Moreover, the corresponding standard i.e.m. $f$ \, is given by
$$
f(x)=x+\omega_{\alpha},\,\,\,\, x\in I_{\alpha},
$$
where
$$
\omega_{\alpha}=\sum\limits_{\pi_{1}(\beta)<\pi_{1}(\alpha)}\lambda_{\beta}- \sum\limits_{\pi_{0}(\beta)<\pi_{0}(\alpha)}\lambda_{\beta}.
$$
We call $\omega=\{\omega_{\alpha}\}_{\alpha\in \mathcal{A}}$ the \textbf{translation vector} of $f$. Note that between length vector $\lambda=(\lambda_{\alpha})_{\alpha\in \mathcal{A}}$ and translation vector $\omega=\{\omega_{\alpha}\}_{\alpha\in \mathcal{A}}$ there is relation: $\omega=\Omega_{\pi}(\lambda)$, where the antisymmetric matrix $(\Omega_{\alpha,\, \beta})_{\alpha,\, \beta\in \mathcal{A}}$ of $\Omega_{\pi}$ is given by
$$
\Omega_{\alpha, \beta}= \left\{\begin{array}{ll}
+1,\,\,\,\,  \mbox{\rm if} \,\,\,\,\, \pi_{1}(\alpha)>\pi_{1}(\beta)\,\,\,\,\, \mbox{\rm and}\,\,\, \pi_{0}(\alpha)<\pi_{0}(\beta), \\
-1,\,\,\,\,  \mbox{\rm if} \,\,\,\,\, \pi_{1}(\alpha)<\pi_{1}(\beta)\,\,\,\,\, \mbox{\rm and}\,\,\, \pi_{0}(\alpha)>\pi_{0}(\beta), \\
\,\, 0,\,\,\,\,\, \mbox{\rm otherwise}.
\end{array}\right.
$$

Denote by $\Pi^{1}$ the set of all possible genus one irreducible combinatorial data $\pi=(\pi_{0},\, \pi_{1})$. The \textbf{Rauzy--Veech cocycle} is the linear cocycle over the Razy-Veech renormalization $R$ defined by
$$
F_{R}: \Pi^{1}\times \mathbb{R}^{\mathcal{A}}\rightarrow \Pi^{1}\times \mathbb{R}^{\mathcal{A}},\,\,\,\,\, (\pi,\, v)\rightarrow (r_{\varepsilon}(\pi),\, \Theta_{\pi,\, \varepsilon}(v)),
$$
where $r_{\varepsilon}(\pi)$ is the combinatorial data of $R(f)$ and linear isomorphism $\Theta=\Theta_{\pi,\, \varepsilon}$ defined by
$$
\Theta_{\alpha,\, \beta}= \left\{\begin{array}{ll}
1,\,\,\,\,  \mbox{\rm if either} \,\,\,\,\, \alpha=\beta\,\,\,\, \mbox{\rm or}\,\,\, (\alpha,\, \beta)=(\alpha(1-\varepsilon),\, \alpha(\varepsilon)), \\
0,\,\,\,\,\, \mbox{\rm otherwise}.
\end{array}\right.
$$
In other words, \,
$\Theta_{\pi,\, \varepsilon}=\mathbb{I}+E_{\alpha(1-\varepsilon)\alpha(\varepsilon)},$
where $\mathbb{I}$ is the unit matrix and $E_{\alpha\beta}$ is the elementary matrix whose only nonzero coefficients is $1$ in position $(\alpha,\, \beta)$. Note that $F^{n}_{R}(\pi,\, v)=(R^{n}(\pi),\, \Theta^{n}_{\pi}(v))$ for all $n\geq 1$, where
$$
\Theta^{n}=\Theta^{n}_{\pi}=\Theta_{\pi^{n-1}}\cdot\cdot\cdot\Theta_{\pi^{\prime}}\Theta_{\pi},\,\,\,\, \mbox{\rm and}\,\,\,\, \pi'=r_{\varepsilon}(\pi).
$$

Note also that the matrixes $\Omega_{\pi}$ and $\Theta_{\pi}$ connected with the relation:
\begin{equation}\label{thetaomega}
\Theta_{\pi,\, \varepsilon}\Omega_{\pi}=\Omega_{\pi'}\Theta^{-1\star}_{\pi,\, \varepsilon},
\end{equation}
where $\Theta^{\star}$ denotes the adjoint operator of $\Theta$, that is, the operator whose matrix is transported of that of $\Theta$.

Let $g:[0, 1]\mapsto [0, 1]$ be an affine i.e.m. without connection. Then $g$ is uniquely determined by the triple $(\pi,\, \lambda,\, \omega^{0})$, where $\pi$ is the combinatorial data, $\lambda=(\lambda_{\alpha})_{\alpha\in \mathcal{A}}\in \mathbb{R}^{d}_{+}$ is the partition vector of the domain and $\omega^{0}=(\omega^{0}_{\alpha})_{\alpha\in \mathcal{A}}\in \mathbb{R}^{d}$ is such that
$$
g(x)=e^{\omega^{0}_{\alpha}}x+\delta_{\alpha},\,\,\,\,  \mbox{\rm for all}\,\,\,\, x\in I_{\alpha}.
$$

For each $n$ denote by $\omega^{n}=(\omega^{n}_{\alpha})_{\alpha}$ the vector such that $R^{n}(g)(x)=e^{\omega^{n}_{\alpha}}x+\delta^{n}_{\alpha}$, for all $x\in I^{(n)}_{\alpha}$. By Rauzy--Veech algorithm we know that
$$
\omega^{n+1}_{\alpha}=\omega^{n}_{\alpha},\,\,\, \mbox{\rm if}\,\,\, \alpha\neq\alpha^{n}(1-\varepsilon),
$$
$$
\omega^{n+1}_{\alpha^{n}(1-\varepsilon)}=\omega^{n}_{\alpha^{n}(\varepsilon)}+\omega^{n}_{\alpha^{n}(1-\varepsilon)},\,\,\, \mbox{\rm otherwise},
$$
where $\alpha^{n}(\varepsilon)$ and $\alpha^{n}(1-\varepsilon)$ are the winners and losers of $R^{n}(f)$, respectively. Therefore, $\Theta_{n}(\omega^{n})=\Theta_{\pi^{n},\, \varepsilon^{n}}(\omega^{n})=\omega^{n+1}$. Repeating the process inductively we have
$$
\Theta_{n}\Theta_{n-1}\cdot\cdot\cdot\Theta_{1}\Theta_{0}(\omega^{0})=\omega^{n}.
$$


It is known that the Rauzy-Veech cocycle is partially hyperbolic [1,3]. To prove Theorem \ref{main5}
we need to understand the hyperbolic properties of the Rauzy--Veech cocycle restricted to the
$k$--bounded combinatorics.

If $\pi$ has genus one then $dim Ker \Omega_{\pi}=d-2$ and $dim Im \Omega_{\pi}=2$. Define the two-dimensional cone \ $C^{s}_{\pi}:=\Omega_{\pi}\mathbb{R}^{\mathcal{A}}_{+}\subset Im \Omega_{\pi}$. It follows from (\ref{thetaomega}) that \ $\Theta^{-1}_{\pi, \varepsilon}C^{s}_{\pi'}\subset C^{s}_{\pi}$. For each $\pi\in \Pi^{1}$ define the convex cone
$$
T^{+}_{\pi}=\left\{(\tau_{\alpha})_{\alpha\in \mathcal{A}}: \sum\limits_{\pi_{0}\leq k}\tau_{\alpha}>0\,\, \mbox{\rm and}\,\, \sum\limits_{\pi_{1}\leq k}\tau_{\alpha}<0,\,\, \mbox{\rm for every}\,\, 1\leq k\leq d-1 \right\}.
$$

Define $C^{u}_{\pi}=-\Omega_{\pi}T^{+}_{\pi}\subset Im \Omega_{\pi}$. Applying $T^{+}_{\pi}$ in (\ref{thetaomega}) we obtain: $\Theta_{\pi,\, \varepsilon}(-\Omega_{\pi}T^{+}_{\pi})=-\Omega_{\pi'}(\Theta^{t}_{\pi,\, \varepsilon})^{-1}T^{+}_{\pi}\subset -\Omega_{\pi'}T^{+}_{\pi'}$. Consequently, we have $\Theta_{\pi,\, \varepsilon}C^{u}_{\pi}\subset C^{u}_{\pi'}$. It is clear that the cones $C^{s}_{\pi}$ and $C^{u}_{\pi}$ are invariant by the action of $\Theta^{-1}_{\pi,\, \varepsilon}$ and $\Theta_{\pi,\, \varepsilon}$, respectively.

The following result \cite{CS2014} shows that Rauzy--Veech cocycle is hyperbolic inside of $\mbox{\it Im}\, \Omega_{\pi}$.

\begin{prop}\label{unfhyper}(Uniform hyperbolicity)(\cite{CS2014}) Let $(\pi^{n},\, \varepsilon^{n})$ be a sequence of $k$-bounded combinatorics with $r_{\varepsilon^{n}}(\pi^{n})=\pi^{n+1}$. Then there exist $\mu=\mu(k)>1$ and $C_{1},\, C_{2}>0$ such that
\begin{itemize}
\item[$(a)$] For every $n$ and $v\in C^{u}_{\pi^{0}}$ we have $\|(\Theta_{\pi^{n},\, \varepsilon^{n}}\cdot\cdot\cdot\Theta_{\pi^{1},\, \varepsilon^{1}}\Theta_{\pi^{0},\, \varepsilon^{0}})(v)\|\geq C_{1}\mu^{n}\|v\|$.
\item[$(b)$] For every $n$ and $v\in C^{s}_{\pi^{n}}$ we have $\|(\Theta_{\pi^{n-1},\, \varepsilon^{n-1}}\cdot\cdot\cdot\Theta_{\pi^{1},\, \varepsilon^{1}}\Theta_{\pi^{0},\, \varepsilon^{0}})^{-1}(v)\|\geq C_{2}\mu^{n}\|v\|$.
\end{itemize}
\end{prop}

\bigskip

Next we study the behaviour of Rauzy--Veech cocycle outside of $\mbox{\it Im}\, \Omega_{\pi}$. According to Proposition \ref{unfhyper}, we define the stable direction in the point $\{\pi^{j},\, \varepsilon^{j}\}$ as
$$
E^{s}_{j}:=E^{s}(\pi^{j})=\bigcap\limits_{n\geq0}\Theta^{-1}_{j}\cdot\cdot\cdot\Theta^{-1}_{j+n}(C^{s}_{\pi^{j+n+1}}).
$$
By definition, the subspaces $E^{s}_{j}$ are invariant by the Rauzy--Veech cocycle, that is, $\Theta_{j}(E^{s}_{j})=E^{s}_{j+1}$, for all $j\geq0$. Now we define the unstable direction. Let $u_{0}\in C^{s}_{\pi^{0}}$ be such that $\|u_{0}\|=1$. Then we define $E^{u}_{0}$ as the subspace spanned by $u_{0}$, that will be denoted by $\langle u_{0}\rangle$. For all $j>0$, we define $E^{u}_{j}:=\langle\frac{u_{j}}{\|u_{j}\|}\rangle$, where $u_{j}=\Theta_{j-1}(u_{j-1})$. The subspaces $E^{u}_{j}$ are forward invariant by the Rauzy--Veech cocycle.

Suppose that the combinatorics has period $p$, that is, there exists $p\in \mathbb{N}$ such that $\{\pi^{n},\, \varepsilon^{n}\}=\{\pi^{n+p},\, \varepsilon^{n+p}\}$ for all $n\in \mathbb{N}$.

\begin{lemm}\label{cdper}(\cite{CS2014}) Define $\Psi_{p}(k)=(\Theta_{0,\, p-1}-Id)^{-1}(k-\Theta_{0,\, p-1})$. Then
\begin{itemize}
\item[(1)] the subspace $E^{c}_{0,\, p-1}:=\{k+\Psi_{p}(k),\,\, k\in Ker(\Omega_{\pi^{0}})\}$ is the central direction of $\Theta_{0,\, p-1}$, that is, $\Theta_{0,\, p-1}v=v$ for every $v\in E^{c}_{0,\, p-1}$.
\item[(2)] the subspace $E^{c}_{0,\, p-1}$ is invariant by Rauzy--Veech cocycle, that is, $\Theta(E^{c}_{0,\, p-1})=E^{c}_{1,\, p-1}$.
\end{itemize}
\end{lemm}

\bigskip

By Proposition \ref{unfhyper} we can choose $n_{0}>0$ and $\mu>>1$ such that
$$
\|\Theta_{n,\, n+n_{0}}(x)\|\geq \mu\|x\|,\,\, \forall x\in C_{\pi^{n}}^{u}\,\,\,\,\, \mbox{\rm and}\,\,\,\,\, \|\Theta_{n,\, n+n_{0}}(x)\|\geq \mu\|x\|,\,\, \forall x\in C_{\pi^{n}}^{s}.
$$

For $\epsilon>0$, define the cones $C^{n}_{\epsilon,\, u}$ and $C^{n}_{\epsilon,\, s}$, where $C^{n}_{\epsilon,\, u}$ is the set of vectors $x=x_{k}+x_{i}\in Ker \Omega_{\pi^{n}}\oplus Im \Omega_{\pi^{n}}$ such that
\begin{itemize}
\item[(a)] $\|x_{k}\|\leq \epsilon \|x_{i}\|$;
\item[(b)] $x_{i}=x^{s}_{i}+x^{u}_{i}$, where $x^{s}_{i}\in \Theta_{n+n_{0}-1,\, n}C^{s}_{\pi^{n+n_{0}}}\subset C^{s}_{\pi^{n}}$,\,\,  $x^{u}_{i}\in \Theta_{n-n_{0},\, n-1}C^{u}_{\pi^{n+n_{0}}}\subset C^{u}_{\pi^{n}}$ and $\|x^{s}_{i}\|\leq x^{u}_{i}$;
\end{itemize}
and analogously defined $C^{n}_{\epsilon,\, s}$ by replacing the last condition by $\|x^{u}_{i}\|\leq x^{u}_{s}$. Define also $C^{n}_{\epsilon}:= C^{n}_{\epsilon,\, s}\bigcup C^{n}_{\epsilon,\, u}$.

\begin{prop}\label{cdsub}(\cite{CS2014}) There exist $\epsilon_{0}=\epsilon_{0}(k)>0$ and $0<\gamma<1$ such that if $\epsilon<\epsilon_{0}$ then
$$
\Theta_{n,\, n+n_{0}-1}C^{n}_{\epsilon,\, u}\subset C^{n+n_{0}}_{\gamma\epsilon,\, u},\,\,\, \mbox{\rm and}\,\,\,  \Theta_{n-1,\, n-n_{0}}C^{n}_{\epsilon,\, s}\subset C^{n-n_{0}}_{\gamma\epsilon,\, u}.
$$
\end{prop}

\begin{prop}\label{eqcon}(\cite{CS2014}) Suppose that $p>n_{0}$. Then $\sup\limits_{0\neq k\in Ker \Omega_{\pi^{0}}}\frac{\|\Psi_{p}(k)\|}{\|k\|}\leq \frac{1}{\epsilon_{0}}$.
\end{prop}

Next we consider arbitrary $k$-bounded combinatorics. Let $f\in \mathbb{B}^{KO}$ and $\gamma(f)=\{\pi^{i},\, \varepsilon^{i}\}_{i\in \mathbb{N}}$ be its combinatorics. For each $n\in \mathbb{N}$ we define the new periodic combinatorics, that will be denoted by $\gamma_{n}(f)=\{\widetilde{\pi}^{i},\, \widetilde{\varepsilon}^{i}\}_{i\in \mathbb{N}}$:
\begin{itemize}
\item[(a)] For $i\leq n$ define $(\widetilde{\pi}^{i},\, \widetilde{\varepsilon}^{i})=(\pi^{i},\, \varepsilon^{i})$, and denote $\widetilde{\gamma}_{n}=\{\widetilde{\pi}^{i},\, \widetilde{\varepsilon}^{i}\}_{i=0}^{n}$;
\item[(b)] Let $\widetilde{\gamma}_{n,\, p_{n}}=\{(\widetilde{\pi}^{n},\, \widetilde{\varepsilon}^{n}),...,(\widetilde{\pi}^{p_{n}},\, \widetilde{\varepsilon}^{p_{n}})\}$ be an admissible sequence of combinatorics, i.e., $r_{\widetilde{\varepsilon}^{i}}(\widetilde{\pi}^{i})=\widetilde{\pi}^{i+1}$ for all $n\leq i <p_{n}$ with $(\widetilde{\pi}^{p_{n}},\, \widetilde{\varepsilon}^{p_{n}})=(\pi^{0},\, \varepsilon^{0})$. It is possible to get this sequence by \cite{Viana}.
\end{itemize}

Then define $\gamma_{n}(f)=(\widetilde{\gamma}_{n}\ast\widetilde{\gamma}_{p_{n}})\ast(\widetilde{\gamma}_{n}\ast\widetilde{\gamma}_{p_{n}})...$. Note that the combinatorics $\gamma_{n}(f)$ is periodic of period $p_{n}$ and that $\gamma_{n}(f)\rightarrow\gamma(f)$ when $n\rightarrow\infty$. The Rauzy--Veech cocycle associated to $\gamma_{n}(f)$ will be denoted by $\widetilde{\Theta}$. By Lemma \ref{cdper} we have that for all $s\geq 0$ the subspace $E^{c}_{s,\, p_{n}}$ is the graph of $\Psi_{s,\, p_{n}}$ and $\widetilde{\Theta}(E^{c}_{s,\, p_{n}})E^{c}_{s+1,\, p_{n}}$. By Proposition \ref{eqcon} the sequence $\{\Psi_{0,\, p_{n}}\}_{n\in \mathbb{N}}$ is equicontinuous and uniformly bounded, so it admits a subsequence $\{\Psi_{0,\, p_{n}}\}_{n\in \mathbb{N}_{0}}$ that uniformly converges. The same holds for $\{\Psi_{1,\, p_{n}}\}_{n\in \mathbb{N}_{0}}$, that is, we can find an infinite subset $\mathbb{N}_{1}\subset \mathbb{N}_{0}$ such that $\{\Psi_{1,\, p_{n}}\}_{n\in \mathbb{N}_{1}}$ is uniformly converge. Proceeding analogously for each $j\in \mathbb{N}$ we can find an infinite subset $\mathbb{N}_{j}\subset \mathbb{N}$ such that $\mathbb{N}_{0}\supset\mathbb{N}_{1}\supset\cdot\cdot\cdot\supset\mathbb{N}_{j}\supset\cdot\cdot\cdot$ and $\{\Psi_{j,\, p_{n}}\}_{n\in \mathbb{N}_{j}}$ uniformly converge. Now we define infinite set $\widetilde{\mathbb{N}}\subset\mathbb{N}$ taking our $j$-th element as $j$-th element of $\mathbb{N}_{j}$. Define the subspace $E^{c}_{j,\, \infty}$ as the graph of $\Psi_{j,\, \infty}=\lim\limits_{\mathbb{N}\ni n\rightarrow\infty}\Psi_{j,\, p_{n}}$.

The next proposition shows that the subspace $E^{c}_{j,\, \infty}$ are invariant by the Rauzy--Veech cocycle and the cocycle is quasi-isomerty in this central direction.

\begin{prop}\label{quasi}(\cite{CS2014})
\begin{itemize}
\item[(1)] For all $j\geq 0$ we have $\Theta_{j}(E^{c}_{j,\, \infty})=E^{c}_{j+1,\, \infty}$;

\item[(2)] For all vectors $v\in E^{c}_{0,\, \infty}$ and for all $n\geq 0$, there is  $C>1$ such that
$$
C^{-1}\|v\|\leq \|\Theta_{0,\, n}v\|\leq C\|v\|,
$$
where $\Theta_{0,\, n}=\Theta_{n}\Theta_{n-1}\cdot\cdot\cdot\Theta_{0}$.
\end{itemize}
\end{prop}

\section{Proof of Main Results}

Let $f:[0, 1)\mapsto[0, 1)$ be a g.i.e.m. For simplicity we write $R^{n}(f)(x)=f_{n}(x)=f^{q^{\alpha}_{n}}(x)$ for $x\in I^{n}_{\alpha}$. Define $L^{n}=(L^{n}_{\alpha})_{\alpha\in \mathcal{A}}$ by
$$
L^{n}_{\alpha}=\frac{1}{|I^{n}_{\alpha}|}\int\limits_{I^{n}_{\alpha}}\ln Df_{n}(s)ds=\frac{1}{|I^{n}_{\alpha}|}\int\limits_{I^{n}_{\alpha}}\ln Df^{q^{\alpha}_{n}}(s)ds.
$$

Note that if $f$ is an affine i.e.m. then $L^{n}_{\alpha}=\omega^{n}_{\alpha}$ for all $\alpha\in \mathcal{A}$. The following proposition gives a relationship between $L^{n}$ and $L^{n+1}$, more precisely we prove that $L^{n}$ is an asymptotic pseudo-orbit for the Kontsevich-Zorich cocycle.

\begin{prop}\label{pseudo-orbit} Let $f\in \mathcal{B}^{KO}_{\star}$. Then
$$
L^{n+1}=\Theta_{n}L^{n}+\overrightarrow{\epsilon}_{n},
$$
where ${\|{\overrightarrow{\epsilon}_{n}}\|}=O(\delta_{n})$ with $\{\delta_{n}\}\in l_2$.
\end{prop}
\begin{proof} Denote by $x^{n}_{\alpha}\in I^{n}_{\alpha}$ the point such that $L^{n}_{\alpha}=\ln Df^{q^{\alpha}_{n}}(x^{n}_{\alpha})$, for all $n\geq 0$.
\begin{itemize}
\item[(*)] If \, $\alpha\neq \alpha^{n}(\varepsilon), \alpha^{n}(1-\varepsilon)$ then clearly $L^{n+1}_{\alpha}=L^{n}_{\alpha}$.

\item[(**)] If \, $\alpha= \alpha^{n}(\varepsilon)$ then $q^{n}_{\alpha^{n}(\varepsilon)}=q^{n+1}_{\alpha^{n}(\varepsilon)}$ and therefore
$$
L^{n+1}_{\alpha}=\ln Df^{q^{n}_{\alpha}}(x^{n+1}_{\alpha})=L^{n}_{\alpha}+\ln Df^{q^{n}_{\alpha}}(x^{n+1}_{\alpha})-\ln Df^{q^{n}_{\alpha}}(x^{n}_{\alpha})=L^{n}_{\alpha}+O(\delta_{n}).
$$
\item[(***)] If \, $\alpha=\alpha^{n}(1-\varepsilon)$ then $q^{n}_{\alpha^{n+1}(1-\varepsilon)}=q^{n}_{\alpha^{n}(1-\varepsilon)}+q^{n}_{\alpha^{n}(\varepsilon)}$. Note also that $f^{q^{n}_{\alpha(1-\varepsilon)}}\in I^{n}_{\alpha^{n}(\varepsilon)}$. Therefore
\end{itemize}
$$
L^{n+1}_{\alpha^{n}(1-\varepsilon)}=\ln Df^{q^{n+1}_{\alpha^{n+1}(1-\varepsilon)}}(x^{n+1}_{\alpha^{n}(1-\varepsilon)})=\ln Df^{q^{n}_{\alpha^{n}(\varepsilon)}}(f^{q^{n}_{\alpha^{n}(1-\varepsilon)}}(x^{n+1}_{\alpha^{n}(1-\varepsilon)}))+
$$
$$
+\ln Df^{q^{n}_{\alpha^{n}(1-\varepsilon)}}(x^{n+1}_{\alpha^{n}(1-\varepsilon)})=
\ln Df^{q^{n}_{\alpha^{n}(\varepsilon)}}(f^{q^{n}_{\alpha^{n}(1-\varepsilon)}}(x^{n+1}_{\alpha^{n}(1-\varepsilon)}))-\ln Df^{q^{n}_{\alpha^{n}(\varepsilon)}}(x^{n}_{\alpha^{n}(\varepsilon)})+
$$
$$
+L^{n}_{\alpha^{n}(\varepsilon)}+\ln Df^{q^{n}_{\alpha^{n}(1-\varepsilon)}}(x^{n+1}_{\alpha^{n}(1-\varepsilon)})- \ln Df^{q^{n}_{\alpha^{n}(1-\varepsilon)}}(x^{n}_{\alpha^{n}(1-\varepsilon)})+L^{n}_{\alpha^{n}(1-\varepsilon)}=
$$
$$
=
L^{n}_{\alpha^{n}(\varepsilon)}+L^{n}_{\alpha^{n}(1-\varepsilon)}+O(\delta_{n}).
$$
This completes the proof.
\end{proof}

Now we decompose the vector $L_{n}=(L_{n})_{\alpha\in \mathcal{A}}$ as
$$
L_{n}=L^{s}_{n}+L^{c}_{n}+L^{u}_{n}\in E^{s}_{n}\oplus E^{c}_{n, \infty}\oplus E^{u}_{n}.
$$

\begin{lemm}\label{lstab}
The sequence $\{L^{s}_{n}\}_{n\in N}$ satisfies $\|L^{s}_{n}\|=O(\delta_{n})$ with $\{\delta_{n}\}\in l_2$.
\end{lemm}
\begin{proof} By Proposition \ref{unfhyper} we have for all $j,\, n\geq 0$ and for all $v\in E^{s}_{j}$ that
$$
\|\Theta_{j+n-1}\Theta_{j+n-2}\cdot\cdot\cdot\Theta_{j}v\|\leq \frac{1}{C_{2}\cdot \mu^{n}}\|v\|.
$$
Replacing this norm by the \textit{adapted norm}, see [13, Proposition 4.2], that we still denote it by $\|\cdot\|$ for simplicity, we can find $\mu>\widetilde{\mu}>1$ such that for all $n\geq 0$ and for all $v\in E^{s}_{j}$ we have $\|\Theta_{n}v\|\leq \widetilde{\mu}^{-1}\|v\|$. By Proposition \ref{pseudo-orbit} we have
$$
\|L^{s}_{n}\|\leq \frac{1}{\widetilde{\mu}}\|L^{s}_{n-1}\|+C\cdot \delta_{n}.
$$
Applying this estimate $n$ times, we obtain
$$
\|L^{s}_{n}\|\leq \frac{1}{\widetilde{\mu}^{n}}\|L^{s}_{0}\|+C\cdot \sum\limits_{i=0}^{n-1}\frac{1}{\widetilde{\mu}^{i}}\delta_{n-i-1}.
$$
Since $\widetilde{\mu}>1$, by the Lemma \ref{extra} we get the claim.
\end{proof}

\begin{lemm}\label{lunstab}
The sequence $\{L^{u}_{n}\}_{n\in N}$ satisfies $\|L^{u}_{n}\|=O(\delta_{n})$ with $\{\delta_{n}\}\in l_2$.
\end{lemm}
\begin{proof} The proof is similar to Lemma \ref{lstab} and we use the adapted norm again. For all $n\geq 0$ we have that
$$
\|L^{u}_{n+1}\|\geq \widetilde{\mu}\|L^{u}_{n}\|-C\cdot \delta_{n}.
$$
Applying this estimate $k$ times, we obtain
$$
\|L^{u}_{n+k}\|\geq \widetilde{\mu}^{k}\|L^{u}_{n}\|-C\cdot \sum\limits_{j=0}^{k-1}\widetilde{\mu}^{j}\delta_{n+k-i-1},
$$
and therefore
$$
\|L^{u}_{n}\|\leq \frac{1}{\widetilde{\mu}^{k}}\|L^{u}_{n+k}\|+C\cdot \sum\limits_{j=0}^{k-1}\widetilde{\mu}^{j-k}\delta_{n+k-j-1},
$$
Taking $n=k$, we have
$$
\|L^{u}_{n}\|\leq \frac{1}{\widetilde{\mu}^{n}}\|L^{u}_{2n}\|+C\cdot \sum\limits_{j=0}^{n-1}\widetilde{\mu}^{j-n}\delta_{2n-j-1}.
$$
Since the sequence $\{L^{u}_{2n}\}$ is uniformly bounded and $\widetilde{\mu}>1$, by the Lemma \ref{extra} we get the claim.
\end{proof}

Define $\widetilde{\omega}_{n}:=\Theta_{0}^{-1}\Theta_{1}^{-1}\cdot\cdot\cdot\Theta_{n-1}^{-1}(L^{c}_{n})\in E^{c}_{0, \infty}$. By Proposition \ref{quasi} and Proposition \ref{pseudo-orbit} imply that $\|\widetilde{\omega}_{n+1}-\widetilde{\omega}_{n}\|=O(\delta_{n})$ with $\{\delta_{n}\}\in l_2$. Therefore $\{\widetilde{\omega}_{n}\}$ converges.

\begin{lemm}\label{limvector} Let \, $\omega=\lim\limits_{n\rightarrow\infty}\widetilde{\omega}_{n}\in E^{c}_{0, \infty}$ and $\omega^{n}\in E^{c}_{0, \infty}$ be the orbit of $\omega$ by Rauzy-Veech cocycle, that is, $\omega^{n}=\Theta_{0}\cdot\cdot\cdot\Theta_{n-1}\omega$. Then $\|\omega^{n}-L_{n}\|=O(\delta_{n})$ with $\{\delta_{n}\}\in l_2$.
\end{lemm}
\begin{proof} By Lemmas \ref{lstab} and \ref{lunstab}, it is sufficient to estimate $\omega^{n}-L^{c}_{n}$:
$$
\|\omega^{n}-L^{c}_{n}\|=\|\Theta_{0}\cdot\cdot\cdot\Theta_{n-1}\omega-L^{c}_{n}\|\leq
$$
$$
\leq\|\Theta_{0}\cdot\cdot\cdot\Theta_{n-1}|_{E^{c}_{0,\infty}}\|\cdot\|\omega-(\Theta_{0}\cdot\cdot\cdot\Theta_{n-1})^{-1}L^{c}_{n}\|\leq C(1+\frac{1}{\epsilon_{0}})\cdot\|\omega-\widetilde{\omega}_{n}\|=O(\delta_{n}).
$$
\end{proof}

Let $f_{A}$ be an affine i.e.m. Denote by $\zeta^{n}$ and $\widetilde{\zeta}^{n}$ the partition vectors of $R^{n}(f)$ and $R^{n}(f_{A})$, respectively.

\begin{lemm}\label{difpartition} Let $f\in \mathcal{B}^{KO}_{\star}$ and $f_{A}$ be affine model. Then \, $|\zeta^{n}-\widetilde{\zeta}^{n}|=O(\delta_{n})$ with $\{\delta_{n}\}\in l_2$.
\end{lemm}

\begin{lemm}\label{uniqaffine}(see, \cite{MMY2010}) Let $\omega$ be a slope vector. Then there exists unique affine i.e.m. $f_{A}$ with domain $[0, 1]$, whose combinatorics is $\{\pi^{i}, \varepsilon^{i}\}_{i\in N}$ and the slope vector is $\omega$.
\end{lemm}

\begin{rem} For each $\omega$ that is a sum of the vector given by Lemma \ref{limvector} and a vector in $E^{s}_{0}$ we
constructed the unique affine interval exchange map $f_{A}$ given by Lemma \ref{uniqaffine}. Each one of
these affine interval exchange maps is called a \textit{weak affine model} of $f$. Note also that the weak affine model is not unique.
\end{rem}

From now on we assume without loss of generality that $f$ has only one discontinuity which will be denoted by $\partial{I_{\alpha^{\star}}}$. The following Lemma states that $f$ and $f_{A}$ have $d-2$ identical breaks.

\begin{lemm}\label{breakaffine} Suppose that $f\in \mathcal{B}^{KO}_{\star}$. Let $f_{A}$ be a week affine model of $f$. Then $BP_{f}(\partial{I_{\alpha^{\star}}})=BP_{f_{A}}(\partial{\widetilde{I}_{\alpha^{\star}}})$ for all $\alpha\in \mathcal{A}$ such that $\alpha\neq \alpha^{\star}$ and $\pi_{0}>1$.
\end{lemm}

The proofs of the Lemmas \ref{difpartition} and \ref{breakaffine} are similar as analogues ones in \cite{CS2014}.

\begin{lemm}\label{distancepartition} Let $f\in \mathcal{B}^{KO}_{\star}$ and $f_{A}$ be a week affine model of $f$. Then $\|R^{n}f(I^{n}_{\alpha})-R^{n}f_{A}(\widetilde{I}^{n}_{\alpha})\|_{C^{1+L_{1}}}=O(\delta_{n})$ for all $\alpha\in \mathcal{A}$.
\end{lemm}

Lemma \ref{distancepartition} estimates the distance between the image partition vector of $R^{n}(f)$ and $R^{n}(f_{A})$. The proof of the Lemma \ref{distancepartition} follows from Lemma \ref{limvector} and Lemma \ref{difpartition}.

\textbf{Proof of Theorem \ref{main3}}. Note that by Theorem \ref{main1} we have
$$
\|Z_{I^{n}_{\alpha}}R^{n}f-Z_{\widetilde{I}^{n}_{\alpha}}R^{n}f_{A}\|_{C^{1+L_{1}}}=O(\delta_{n})\,\,\, \mbox{\rm for all}\,\, n\geq0\,\,\, \mbox{\rm and}\,\,\, \alpha\in \mathcal{A}\,\,\, \mbox{\rm with}\,\,\, \{\delta_{n}\}\in l_{2}.
$$
The last equality together with Lemma \ref{difpartition} and Lemma \ref{distancepartition} imply Theorem \ref{main3}.

\textbf{Proof of Theorem \ref{main4}}. For simplicity we assume $\mathcal{A}=\{1, 2,...,d\}$ and denote by $j_{0}\in \mathcal{A}$ the letter such that $\partial{I_{j_{0}}}$ is the discontinuity of $f_{A}$. Since $f$ and $g$ are break-equivalent, by Lemma \ref{breakaffine} we have
$$
\omega^{f}_{i+1}-\omega^{f}_{i}=\omega^{g}_{i+1}-\omega^{g}_{i},\,\,\,\, \mbox{\rm for every} \,\, i\in \mathcal{A}\,\,\,\ such that i\neq j_{0}-1,\,d,
$$
which is equivalent to
$$
\omega^{f}_{1}-\omega^{g}_{1}=\cdot\cdot\cdot=\omega^{f}_{j_{0}-1}-\omega^{g}_{j_{0}-1},\,\,\, \omega^{f}_{j_{0}}-\omega^{g}_{j_{0}}=\cdot\cdot\cdot=\omega^{f}_{d}-\omega^{g}_{d},
$$
where we choose $\omega^{f}=\{\omega^{f}_{i}\}_{i\in \mathcal{A}}\in E^{c}_{0, \infty}$ and $\omega^{g}=\{\omega^{g}_{i}\}_{i\in \mathcal{A}}\in E^{c}_{0, \infty}$ as the slope-vector of the week affine models $f_{A}$ and $g_{A}$, respectively.

Denote $v:=\omega^{f}_{1}-\omega^{g}_{1}$ and $\widetilde{v}:=\omega^{f}_{j_{0}}-\omega^{g}_{j_{0}}$. Then we have that
$$
\omega^{f}-\omega^{g}=(\omega^{f}_{i}-\omega^{g}_{i})_{i\in \mathcal{A}}=(v,...v, \underbrace{\widetilde{v}}\limits_{j_{0}\,\, \mbox{\rm position}},...,\widetilde{v})\in E^{c}_{0, \infty},
$$
that is, the vector $\omega^{f}-\omega^{g}$ can be viewed as the slope-vector of an affine interval exchange maps with two break intervals. So we have $(v,\widetilde{v})\in E^{c}_{0, \infty}(2)$, where $E^{c}_{0, \infty}(2)$ is the central space defined by renormalization of two intervals. Since $dimE^{c}_{0, \infty}(2)=0$ we have $v=\widetilde{v}=0$ and then $\omega^{f}=\omega^{g}$. By Proposition \ref{uniqaffine} we get $f_{A}=g_{A}$ as claimed.

\textbf{Proof of Theorem \ref{main5}}. Let $f$ and $g$ be as in the assumptions of Theorem \ref{main5}. Then by Theorem \ref{main4}, we have: $f_{A}=g_{A}$. Therefore,
$$
\|Z_{I^{n}_{\alpha}}R^{n}f-Z_{\widetilde{I}^{n}_{\alpha}}R^{n}g\|\leq \|Z_{I^{n}_{\alpha}}R^{n}f-Z_{\widetilde{I}^{n}_{\alpha}}R^{n}f_{A}\|+
\|Z_{I^{n}_{\alpha}}R^{n}g-Z_{\widetilde{I}^{n}_{\alpha}}R^{n}g_{A}\|=O(\delta_{n}).
$$
This completes the proof of Theorem \ref{main5}.

\end{document}